\RequirePackage{fix-cm}
\documentclass[smallextended]{svjour3}       
\smartqed  
\usepackage{amssymb}
\usepackage{amsmath}
\newcommand{\GKdim}{\mathrm{GKdim}}
\newcommand{\supp}{\mathrm{supp}}
\newcommand{\PI}{\mathrm{PI}}

\newcommand{\Aut}{\mathrm{Aut}}
\newcommand{\N}{\mathbb{N}}
\newcommand{\Z}{\mathbb{Z}}
\newcommand{\F}{\mathbb{F}}
\newcommand{\C}{\mathbb{C}}
\newcommand{\Q}{\mathbb{Q}}

\begin{document}
 \title{On weakly locally finite division rings\thanks{This work is funded by Vietnam National Foundation for Science and Technology Development (NAFOSTED) under Grant No. 101.04-2016.18.}
}
\author{Trinh Thanh Deo\and Mai Hoang Bien\and Bui Xuan Hai}

\authorrunning{T.T. Deo, M. H. Bien, B.X.Hai} 

\institute{Trinh Thanh Deo\at
Faculty of Mathematics and Computer Science, VNUHCM - University of Science,\\ 227 Nguyen Van Cu Str., Dist. 5, Ho Chi Minh City, Vietnam.\\
\email{ttdeo@hcmus.edu.vn}           
\and 
Mai Hoang Bien\at
Faculty of Mathematics and Computer Science, VNUHCM - University of Science,\\ 227 Nguyen Van Cu Str., Dist. 5, Ho Chi Minh City, Vietnam.\\
\email{mhbien@hcmus.edu.vn}           
\and
Bui Xuan Hai\at
Faculty of Mathematics and Computer Science, VNUHCM - University of Science,\\ 227 Nguyen Van Cu Str., Dist. 5, Ho Chi Minh City, Vietnam.\\
\email{bxhai@hcmus.edu.vn}           
}
\date{Received: date / Accepted: date}
\maketitle
 \begin{abstract} 
Weakly locally finite division rings were considered in \cite{dbh}, where it was mentioned that the class of weakly locally finite division rings properly contains  the class of locally finite division rings. In this paper, for any integer  $n\geq 0$ or $n=\infty$, we construct a weakly locally finite  division ring whose Gelfand-Kirillov dimension is $n$. This fact shows in particular that there exist infinitely many weakly locally finite division rings that are not locally finite. Further, for the class of weakly locally finite division rings, we investigate some questions related with the well-known Kurosh Problem and with one of Herstein's conjectures. 
\end{abstract}

\section{Introduction}
In the theory of algebras, it is well-known that an algebra is locally finite iff its Gelfand-Kirillov dimension ($\GKdim$ for short) is $0$ \cite{kra-len}. Recall that the class of weakly locally finite division rings considered in \cite{dbh} is a natural generalization of the class of  locally finite division rings.  As we will see in the text, a division ring is weakly locally finite iff it is PI. In 1996, Zhang \cite[Example 5.7]{zhang} gave the example of a locally PI division ring whose $\GKdim$~is~$2$. Therefore, the class of weakly locally finite division rings properly contains the class of locally finite division rings. In this paper, using Zhang's idea, we construct the example of a locally PI division ring with $\GKdim=n\ge 1$ or $\infty$. 
Here, we use directly Mal'cev Neumann's construction of the division ring of the free abelian group $G$ of countable rank over some suitable base field $K$ with respect to a certain group morphism $\Phi : G\to \Aut(K)$. Hence, we show in particular that there exist infinitely many weakly locally finite division rings that are not locally finite.
Some readers of this manuscript called our attention to an old but unpublished example of  J. C. McConnell on division ring with arbitrary Gelfand-Kirillov dimension (cf. \cite{McConnell}). From our discussion with J. C. McConnell and his colleagues we felt  that it is worth to have more examples on division rings of arbitrary predescribed Gelfand-Kirillov dimension.

Further, we study some questions related with  the Kurosh Problem for division rings. Recall that in 1941, Kurosh \cite[Problem R]{Kurosh} asked if a finitely generated algebraic algebra is necessarily a finite dimensional vector space over a base field. Equivalently, the Kurosh Problem for division rings asked if an algebraic division ring is locally  finite. At the present,  this problem remains still unsolved in general, but it is answered in the affirmative for several special cases of a division ring $D$ with the center $F$.  In particular, it is the case: for $F$ finite \cite{lam}, and for $F$ having only finite algebraic field extensions (in particular for $F$ algebraically closed). For instance, the last case follows from the Levitzki-Shirshov theorem which states that {\em any algebraic algebra of bounded degree is locally finite} (see e.g. \cite{dren}, \cite{kha}). For an  additional information about this problem we refer to \cite{kha}, \cite{smok4} and \cite{zelmanov}. In the present paper, we investigate the Kurosh Problem and some its weaker versions for matrix rings over a division ring. In particular, we prove that the Kurosh Problem is answered in the affirmative for weakly locally finite division rings.

Also, we consider a conjecture posed by Herstein. In fact, in 1978, I.N. Herstein  \cite[Conjecture 3]{her} conjectured that given a subnormal subgroup $N$ of the multiplicative group $D^*$ of a division ring $D$, if $N$ is radical over the center $F$ of $D$, then $N$ is central, i. e. $N$ is contained in $F$. Herstein, himself in the cited above paper proved this fact for the special case, when $N$ is torsion group. However, the problem remains  still open in general. In \cite{hai-huynh}, it was proved that  this  conjecture is  true in the finite dimensional case. In this paper, we  shall prove that this conjecture  is also true for weakly locally finite division rings. 

\section{The Gelfand-Kirillov dimension of an algebra}

Let $A$ be an algebra over a field $k$ and $V$ be a subspace of $A$ containing the identity  $1_A$ of $A$. Assume that  $V$ is generated by elements $a_1, a_2, \ldots, a_n$. For any integer $ r\geq 2$, $V^r$ denotes the subspace of  $A$ generated by all monomials  $a_{i_1} a_{i_2} \ldots a_{i_r}$ of length $r$, where $a_{i_j}\in \{a_1, a_2, \ldots, a_n\}$. If $\sum_{r\geq 1}V^r=A$, then we say that $V$ is a {\em subframe} of $A$. The Gelfand-Kirillov dimension of $A$ over $k$, denoted by $\GKdim_k(A)$, is defined by the following formulae
$$\GKdim_k(A):=\sup\limits_V \overline{\lim\limits_{r\to\infty}}\log_r \dim_k V^r,$$
where $V$ runs over the set of all subframes of $A$. The basic properties of the Gelfand-Kirillov dimension can be found in \cite{kra-len}.
One can show that $A$ is locally finite if and only if $\GKdim_k(A)=0$. Also, it is known that every positive integer can occur as the Gelfand-Kirillov dimension of some commutative algebra. Moreover, for every real number  $r\geq 2$, there always exists some $k$-algebra $A$ with $\GKdim_k(A)=r.$ Bergman \cite{berg} proved that there does not exist any $k$-algebra having Gelfand-Kirillov dimension in the open interval $(1,2)$.

For a division ring $D$, the Gelfand-Kirillov dimension of $D$ is understood to be the Gelfand-Kirillov dimension of an algebra $D$ over its center $F$.

Recall that in \cite{zhang}, Zhang defined the Gelfand-Kirillov transcendence degree of an algebra $A$ over a field $k$ by $$\mathrm{Tdeg}_kA = \sup \limits_V \inf\limits_b \overline{\lim\limits_{n \to \infty}}{\log_n}{\dim_k}({(k + bV)^n}),$$ where $V$ ranges over subframes of $A$ and $b$ ranges over $A^*=A\backslash\{0\}$. 

\section{Weakly locally finite division rings}

Let $D$ be a division ring with center $F$. Recall that $D$ is \emph{centrally finite} if $D$ is a finite dimensional vector space over $F$. If for every finite subset $S$ of $D$, the division subring $F(S)$ generated by $S$ over $F$ is a finite dimensional vector space over $F$ then $D$ is called {\it locally finite}.  We begin  with the  observation that in a centrally finite division ring, every division subring is itself centrally finite. Using this fact, it is easy to show that in a locally finite division ring, every finite subset generates a centrally finite division subring. Motivating by this observation, we have introduced the following notion.

\begin{definition}\label{def:2.1} 
We say that a division ring $D$  is  \emph{weakly locally finite} if  for every finite subset $S$ of $D$, the division subring  generated by $S$ in $D$ is  centrally finite.
\end{definition}

It follows immediately from the definition that a locally finite division ring is weakly locally finite. In the sequent, we shall show that there exist infinitely many weakly locally finite division rings that are not locally finite. Recall that an algebra $A$ over a field $k$ is said to be a \emph{locally $\PI$ algebra} if every finitely generated subalgebra of $A$ is a PI algebra. It turns out that weakly locally finite division rings are exactly locally PI division rings (regarding as algebras over their centers) as the following theorem shows.

\begin{theorem}\label{thm:2.2} A division ring is weakly locally finite if and only if it is \mbox{locally $\PI$.}
\end{theorem}

\begin{proof} Let $D$ be division ring with center $F$.
	
$(\Rightarrow)$ Assume that $D$ is weakly locally finite. For any finite subset $S$ of $D$, we have to prove that the subring $F[S]$ of $D$ generated by $S$ over $F$ is a $\PI$ algebra. Indeed, the division subring $L$ of $D$ generated by $S$ is centrally finite. Let $S=\{s_1, s_2, \ldots, s_t\}$, and  $\mathcal{B}=\{x_1, x_2, \ldots, x_n\}$ be a basis of $L$ over its center $Z(L)$. For any $1\leq i, j\leq t$, write $s_{i}s_j= a_{ij1} x_1 + a_{ij2}x_2 + \ldots + a_{ijn}x_n,$ where $a_{ijk}\in Z(L)$. Then, the  division subring $K$ of $D$ generated by $F$ and all $a_{ijk}$  is a subfield of $D$.
	
Put $H=K[S]$. Then, $H$ is a subring of $D$ containing $F[S]$, the field $K$ is contained in the center of $H$ and $H$ is a finite dimensional vector space over $K$. Hence,  $H$ can be considered as a subring of the matrix ring $M_m(K)$ with $m=\dim_KH$. Since $M_m(K)$ is a PI algebra, $H$ is a PI algebra as well. Therefore, $F[S]$ is a PI algebra.

$(\Leftarrow )$ Assume that $D$ is a locally PI algebra and $S$ is a finite subset of $D$. Then, $F[S]$  is a PI $F$-algebra. By \cite[Theorem 5.6]{zhang}, $F[S]$ is an  Ore domain. In view of the Posner Theorem \cite[Theorem 6.1.11]{bei},  the division subring $F(S)$ of $D$ generated by $S$ over $F$ is centrally finite, so is every its division subring. In particular,  $\langle S\rangle$ is centraly finite. Therefore, $D$ is weakly locally finite.
\end{proof}

From \cite[Theorem 5.6]{zhang} and Theorem \ref{thm:2.2}, it follows that if $D$ is a weakly locally finite division ring with center $F$, then $\mathrm{Tdeg}_FD=\GKdim_FD$. Zhang have noted \cite[Page 2871]{zhang} that it is unknown if there exists a division ring $D$ with center $F$ such that  $\mathrm{Tdeg}_F D$ is a non-integer. The following proposition shows that there are no weakly locally finite division ring whose Gelfand-Kirillov dimension is non-integer.

\begin{proposition} Let $D$ be a weakly locally finite division ring with center $F$. Then,  $\GKdim_FD\in \N$ or $\GKdim_FD= \infty$.
\end{proposition}

\begin{proof} By Theorem \ref{thm:2.2}, $D$ is a locally PI algebra. By a remark in \cite[Page 14]{kra-len}, 
$$\GKdim_FD=\max\{\GKdim_FB\mid B\subseteq D, B \text{ is finitely generated over } F\}.$$ 
Therefore, to prove the proposition, it suffices to show that any $F$-subalgebra $B$ of $D$ generated by a finite subset $S$, the Gelfand-Kirillov dimension $\GKdim_FB$ is an integer. Indeed, since $D$ is a locally PI $F$-algebra, $B$ is a PI $F$-algebra. Now, in view of \cite[Theorem 10.5]{kra-len}, $\GKdim_FB$ is an integer. 
\end{proof}

Further, in this section, for a given $n$ which is either non-negative integer or  $n=\infty$, we construct a weakly locally finite division ring whose $\GKdim$ is $n.$ The case $n=0$ is trivial because there exists a vast number of locally finite division rings.

\begin{remark} 
	
In Subsection 3.2 and Subsection 3.3, we construct the examples of a locally PI  division rings with $\GKdim=n\ge 2$ or $\infty$  respectively. However, this idea does not work for the case of a division ring with $\GKdim=1$. That is the reason why   in Subsection~3.1, we construct separately the example of a division ring with $\GKdim=1$.
\end{remark}

\subsection{Weakly locally finite division ring with the  Gelfand-Kirillov dimension one}

The case $n=1$ is more complicated than the cases $n\geq 2$ or $n=\infty$, so we consider it separately. Now we fix some notation.

Let $k$ be a field of characteristic $p\ne 2$ and $Q=k(\lambda)$ the quotient field of the polynomial ring $k[\lambda]$ in one indeterminate $\lambda$. Denote by $\overline Q$ an algebraic closure of $Q$. It is well known that $k[\lambda]$ contains an infinite set of prime elements. Let $\{p_i\mid i\in \N\}$ be a sequence of distinct prime elements. For any $i\in \N$, the polynomial $t^2-p_i$ are irreducible in $Q[t]$.
Denote by $\sqrt{p_i}$ a root of the algebraic equation $t^2-p_i=0$ in $\overline Q$.  
For $n\in \N$ let $K_n=Q(\sqrt{p_1}, \ldots, \sqrt{p_n})$ and $K=\bigcup_{n\in\N} K_n$. It is well-known by a result of Besicovitch (in \cite{besic}, 1940), rediscovered by both Mordell (in \cite{mordell}, 1953) and Siegel (in \cite{siegel}, 1972) (cf. also website of Math Stack Exchange\footnote{https://math.stackexchange.com/questions/30687/the-square-roots-of-different-primes-are-linearly-independent-over-the-field-of}) that $K_n/Q$ and $K/Q$ are Galois extensions whose Galois groups are isomorphic to $(\Z/2\Z)^n$ and $(\Z/2\Z)^{\N}$, respectively.
Therefore there exists for each $i\in \N$ the $Q$-automorphism $f_i$ of $K$ such that:
$$f_i(\sqrt{p_i})=-\sqrt{p_i};\quad \text{and } \ f_i(\sqrt{p_j})=\sqrt{p_j}\quad \text{for any }j\neq i.$$ 
Moreover, the set 
$$\big\{\sqrt{p_{i_1} p_{i_2} \ldots p_{i_n}}\mid i_1 < i_2<\ldots<i_n, n\in \N\big\}$$
is obviously a $Q$-base of $K$ and $Q$ is the fixed subfield of $K$ under all $f_i=f_i^{-1}, i\in\N$. 

Let $G=\oplus_{n\in \N} \Z$ be the free abelian group of countable infinite rank linearly ordered in the obvious manner. Thus elements of $G$ are multiindices $x=(n_1, n_2,\ldots); n_i \in\Z$ and $x< y=(m_1, m_2, \ldots)\iff n_k<m_k$ if $k$ is the first index such that $n_k\neq m_k$.
Moreover, every element $x=(n_1, n_2,\ldots)$ in $G$ is written uniquely in the form $x=\sum n_i x_i$, where $x_i= (0, \ldots, 0, 1, 0, \ldots)$ is the element of $G$ with  $1$ in the $i$-th position and $0$ elsewhere.

For any  $x=(n_1, n_2, \ldots)=\sum n_i x_i\in G$, define $\Phi: G\to \mathrm{Gal}(K/Q), x\mapsto \Phi_x$ be a group morphism generated by $\Phi_{x_i}:=f_i, i\in \N$. The proof of the following lemma is elementary.

\begin{lemma}\label{l4.2}  	 
\begin{enumerate}
  \item $\Phi(x)=\prod\limits_{i}f_i^{n_i}$. 
  \item $x_i\sqrt{p_i}=-\sqrt{p_i}x_i$ and  $x_j\sqrt{p_i}=\sqrt{p_i}x_j$ for $j\neq i$.
  \item $xa=\Phi_x(a) x$ for any $a\in K$.
\end{enumerate}
\end{lemma}

For the convenience, from now on we write the operation in  $G$ multiplicatively. For $G$ and $K$ as above, consider formal sums of the form
$$\alpha=\sum\limits_{x\in G} a_x x,\quad a_x\in K.$$
For  such an $\alpha$, the subset $\supp(\alpha)=\{x\in G: a_x\neq 0\}$ of $G$ is called the {\it support} of $\alpha$. Put
$$D=K((G,\Phi)):=\Big\{\alpha=\sum\limits_{x\in G} a_x x, a_x\in K \mid \supp(\alpha) \text{ is  well-ordered }\Big\}.$$ 
For  $\alpha=\sum\limits_{x\in G} a_x x$ and $\beta=\sum\limits_{x\in G} b_x x$ from $D$, 
define 
$$\alpha+\beta=\sum\limits_{x\in G} (a_x+b_x) x,$$
and  
$$\alpha\beta=\sum\limits_{z\in G} \Big(\sum\limits_{xy=z}a_x \Phi_x(b_y)\Big) z.$$
Then,  $D=K((G,\Phi))$ is a division ring (we refer to \cite[(14.21), p. 231]{lam} for details). 
The construction of these division rings followed the method raised in \cite{neumann}. Division rings obtained by such a way as in \cite{neumann} are now often referred as Mal'cev-Neumann division rings. The following proposition describes the center of the division ring we just defined above. 

\begin{proposition}\label{p4.3}
Let $D=K((G,\Phi))$ be as above. Put $H=\{x^2\mid x\in G\}$ and 
$$Q((H))=\Big\{\alpha=\sum\limits_{x\in G} a_x x\in K((G,\Phi))\mid \supp(\alpha) \subseteq H \text{ and } a_x\in Q\Big\}.$$ 
Then $Q((H))$ is the center of $D$.
\end{proposition}

\begin{proof}
Since $G$ is abelian, $H$ is a subgroup of $G$. Moreover, it is easy to check that for every $x\in H$, $\Phi_x=Id_K$. 
Denote by $F$ the center of $D$. We have to show that $F=Q((H)).$ Observe that every element $\alpha\in Q((H))$ can be written in the form  $\alpha =\sum\limits_{x\in H} a_x x$, where $a_x\in Q$. Now, suppose that $\alpha=\sum\limits_{x\in H} a_x x \in Q((H))$. Then, for every  $\beta=\sum\limits_{y\in G} b_y y \in D$, we have $\Phi_x (b_y)= b_y$ and $\Phi_y (a_x)= a_x$. Hence
\begin{eqnarray*}
\alpha\beta&=&\sum\limits_{z\in G} \Big(\sum\limits_{xy=z}a_x \Phi_x(b_y)\Big) z
=\sum\limits_{z\in G} \Big(\sum\limits_{xy=z}a_x b_y\Big) z,\\
\beta\alpha&=&\sum\limits_{z\in G} \Big(\sum\limits_{xy=z}b_y \Phi_y(a_x)\Big) z
=\sum\limits_{z\in G} \Big(\sum\limits_{xy=z}a_x b_y\Big) z.
\end{eqnarray*} 
Thus, $\alpha\beta=\beta\alpha$ for every $\beta\in D$, so $\alpha \in F$.

Conversely, suppose that  $\alpha=\sum\limits_{x\in G} a_x x \in F.$
Denote by  $S=\supp(\alpha)$. Then, it suffices to prove that  $x\in H$ and $a_x\in Q$ for any $x\in S$. 
In fact, since $\alpha \in F$, we have $\sqrt{p_i}\alpha=\alpha \sqrt{p_i}$ and $\alpha x_i =x_i\alpha$ for any $i\geq 1$, i.e.
$\sum\limits_{x\in S} \sqrt{p_i}a_x x=\sum\limits_{x\in S} \Phi_x(\sqrt{p_i})a_x x$ and $\sum\limits_{x\in S} a_x (x x_i)=\sum\limits_{x\in S} \Phi_{x_i}(a_x) (x_i x).$
Therefore, by Lemma~\ref{l4.2}, we have $\sqrt{p_i}a_x =\Phi_x(\sqrt{p_i})a_x=(-1)^{n_i}\sqrt{p_i}a_x$ and $a_x=\Phi_{x_i}(a_x)=f_i(a_x)$ for any $x=x_1^{n_1}x_2^{n_2}\ldots x_t^{n_t}\in S$.
From the first equality it follows that  $n_i$ is even for any  $i\geq 1$. Therefore $x\in H$. From the second equality it follows that $a_x=f_i (a_x)$ for any $i\geq 1$, hence $a_x\in Q$. Therefore $\alpha \in Q((H)).$
Thus,  $F=Q((H)).$
\end{proof}

\begin{lemma}\label{l4.4} Let $D=K((G,\Phi))$ be as above and let $\gamma = x_1^{-1}+x_2^{-1}+\ldots$ be an infinite formal sum. Then $\gamma$ is an element of $D$ which is not algebraic over the center $F=Z(D)$.
\end{lemma}
\begin{proof}
By Proposition~\ref{p4.3}, $F=Q((H))$. Since  $x_1^{-1}<x_2^{-1}<\ldots,  \supp(\gamma)$ is  a well-ordered subset of $G$, hence $\gamma\in D$. To prove that $\gamma$ is not algebraic over $F$, it suffices to show that  $\{\gamma^i\mid i\ge 0\}$ is an independent set over $F$. Indeed, assume that $\{\gamma^i\mid i\ge 0\}$ is dependent over $F$,  and $n$ is the smallest integer such that there exist $a_1,a_2,\ldots,a_n\in F$ with $a_n\ne 0$ and 
$$a_0+a_1\gamma+a_2\gamma^2+\ldots+a_n\gamma^n=0.\eqno(2)$$
Note that $X=x_1^{-1}x_2^{-1}\ldots x_n^{-1}$ does not appear in the expressions of 
$$a_0, a_1\gamma, a_2\gamma^2, \ldots, a_{n-1}\gamma^{n-1}$$
and the coefficient of  $X$ in the expression of $\gamma^n$ is  $n!$. Therefore, from the equality $(2)$, $a_n.n!X=0$. It follows that  $a_n=0$, a contradiction, and the proof of the lemma is now complete.
\end{proof}

Now, we show that $D=K((G,\Phi))$ contains a division subring which is weakly locally finite and its Gelfand-Kirillov dimension is $1$.
\begin{theorem}\label{t4.5} Let $D=K((G,\Phi))$ and $\gamma$ be as in Lemma~\ref{l4.4}. For any $n\geq 1$, denote by 
$$R_n=F(\sqrt{p_1},\sqrt{p_2},\ldots, \sqrt{p_n}, x_1, x_2,\ldots, x_n, \gamma),$$
and  $R_\infty=\bigcup\limits_{n=1}^\infty R_n,$ where $F$ is the center of $D$. Then, $R_\infty$ is a weakly locally finite division ring and $\GKdim(R_\infty)=1$.
\end{theorem}
\begin{proof} 
By Proposition~\ref{p4.3}, $F=Q((H))$. First, we prove that $R_n$ is centrally finite for each positive integer $n$.
Consider the element  $$\gamma_n=x_{n+1}^{-1}+x_{n+2}^{-1}+\ldots \quad \text{(infinite formal sum)}.$$ 
Since $\gamma_n =\gamma -(x_1^{-1}+x_2^{-1}+\ldots+x_n^{-1})$, we conclude that $\gamma_n\in R_n$ and 
$$F(\sqrt{p_1},\sqrt{p_2},..., \sqrt{p_n},x_1,x_2,... , x_n, \gamma)=F(\sqrt{p_1},\sqrt{p_2},..., \sqrt{p_n},x_1,x_2,..., x_n, \gamma_n).$$
Note that $\gamma_n$ commutes with all  $\sqrt{p_i}$ and all $x_i$ (for $i=1,2,...,n$).  Therefore
\begin{eqnarray*}
R_n&=&F(\sqrt{p_1},\sqrt{p_2},\ldots, \sqrt{p_n},x_1,x_2,\ldots,x_n, \gamma_n)\\
&=&F(\gamma_n)(\sqrt{p_1},\sqrt{p_2},\ldots, \sqrt{p_n},x_1,x_2,\ldots, x_n).
\end{eqnarray*}
In view of the equalities  
$$(\sqrt{p_i})^2=p_i, x_i^2\in F, \, \sqrt{p_i}x_j=x_j\sqrt{p_i}, i\ne j,\, \sqrt{p_i}x_i=-x_i\sqrt{p_i},$$
it follows that every element $\beta$ from  $R_n$ can be written in the form 
$$\beta = \sum\limits_{0\leq \varepsilon_i, \mu_i\leq 1} a_{(\varepsilon_1, ..., \varepsilon_n,  \mu_1, ..., \mu_n)} (\sqrt{p_1})^{\varepsilon_1}\ldots (\sqrt{p_n})^{\varepsilon_n} x_1^{\mu_1} \ldots x_n^{\mu_n},$$
where $ a_{(\varepsilon_1, ..., \varepsilon_n\mu_1, ..., \mu_n)}\in F(\gamma_n).$ Hence $\beta$ is in the center of $R_n$ if and only if $\beta\in F(\gamma_n)$. This means that $F(\gamma_n)$ is the center of $R_n$. Moreover,  $R_n$ is a vector space over $F(\gamma_n)$ having the finite set $B_n$ which consists of the products 
$$(\sqrt{p_1})^{\varepsilon_1}\ldots (\sqrt{p_n})^{\varepsilon_n} x_1^{\mu_1} \ldots x_n^{\mu_n}, 0\le \varepsilon_i, \mu_i\le 1$$
as a basis. Thus, $R_n $ is centrally finite. For any finite subset $S\subseteq R_\infty$, there exists  $n$ such that $S\subseteq R_n$. Therefore, the division subring of $R_\infty$ generated by $S$ over $F$ is contained in $R_n$, which is centrally finite. Thus, $R_\infty$ is weakly locally finite. Now, we claim that the center $Z(R_\infty)$ of $R_\infty$ is $F$. Since $F\subseteq Z(R_\infty)$, it suffices to show that $Z(R_\infty)\subseteq F$. Indeed, if $a\in Z(R_\infty)$ then $a$ commutes with all $x_i$ and $\sqrt{p_i}$ for all $i\ge 1$. Hence, $a$ commutes with all elements of $G$ and $K$, which implies that $a\in F$. 
	
Now, we prove that $\GKdim(R_\infty)=1$. By \cite[Lemma 5.4]{zhang}, it suffices to prove that $\GKdim_FR'_n=1$,  where $R'_n=F[\gamma, \sqrt{p_1},\ldots, \sqrt{p_n},x_1,\ldots, x_n]$ is the subring of $R_n$ generated by 
$$\gamma, \sqrt{p_1},\ldots, \sqrt{p_n},x_1,\ldots, x_n$$ 
over $F$. Let $$V=\langle 1, \gamma, \sqrt{p_1},\ldots, \sqrt{p_n},x_1,\ldots, x_n \rangle_F$$ be the vector subspace of $R'_n$ generated by elements $1$, $\gamma$, $\sqrt{p_1}$, $\ldots$, $\sqrt{p_n}$, $x_1$, $\ldots$, $x_n$ over $F$. Then $\sum_{r\ge 1}V^r=R'_n$, which implies that $V$ is a subframe of $R'_n$. We claim that $\dim_FV^r=M+4^nr$ for some $M$ and $r$ sufficiently large. Indeed, for any $r>2n$, in view of  the relations between  $\sqrt{p_i}$ and $x_j$ for any $1\le i,j\le n$: $x_i\sqrt{p_i}=-\sqrt{p_i}x_i$ and $x_j\sqrt{p_i}=\sqrt{p_i}x_j$ if $i\ne j$ (see Lemma~\ref{l4.2}), if  we add the following elements 
$\gamma^{r+1}$, $\gamma^r\sqrt{p_1}$, $\ldots$, $\gamma^r\sqrt{p_n}$, $\gamma^rx_1$, $\ldots$, $\gamma^rx_n$, $\gamma^{r-1}\sqrt{p_1p_2}$, $\ldots$, $\gamma^{r-2n}\sqrt{p_1}\ldots \sqrt{p_n}x_1\ldots x_n$
to some basis of $V^r$, then we obtain a basis of $V^{r+1}$.
Since the number of  added elements  is $C^0_{2n}+C^1_{2n}+\ldots+C^{2n}_{2n}=2^{2n}=4^n,$ one has $\dim_FV^r=M+4^nr$ for some integer $M$. Hence,  $\overline{\lim\limits_{r \to \infty}}{\log_r}{\dim_F}{V^r} = \overline{\lim\limits_{r \to \infty}}{\log_r}{(M+4^nr)}=1$. Now using the remark in \cite[page 14]{kra-len}, we have $\GKdim_FR'_n=\overline{\lim\limits_{r \to \infty}}{\log_r}{\dim_F}{V^r}=1$. Thus, the proof of the theorem is now complete.
\end{proof}

\subsection{Weakly locally finite division ring with the  Gelfand-Kirillov dimension $n\geq 2$}

Throughout this subsection, $k=\C$ is the field of complex numbers, $n$ is a given positive integer and $p$ is a prime number. 
For  positive integers $t$, we construct a sequence of $k$-algebras $A_{nt}$ as the following: 

Let $\{\,x_{11},x_{2t},\ldots, x_{nt}\,\}$ be $n$ non-commutative indeterminates and $r_t\in \C$ the primitive $p^{2t}$-th root of unity such that $r_{t}=r_{t+1}^{p^2}$. Consider  
$$A_{nt}=k\langle \,x_{1t},x_{2t},\ldots, x_{nt}\,\rangle/\langle \,x_{it}x_{jt}-r_t x_{jt}x_{it}\mid i> j\,\rangle,$$
where $k\langle \,x_{11},x_{2t},\ldots, x_{nt}\,\rangle$ is the free algebra in $\{\,x_{11},x_{2t},\ldots, x_{nt}\,\}$ over $k$ and $\langle \,x_{it}x_{jt}-r_t x_{jt}x_{it}\mid i>j\,\rangle$ is the ideal in $k\langle \,x_{11},x_{2t},\ldots, x_{nt}\,\rangle$ generated by all $x_{it}x_{jt}-r x_{jt}x_{it}$ with $i>j$. For an element $\alpha\in k\langle \,x_{11},x_{2t},\ldots, x_{nt}\,\rangle$, the symbol $\overline{\alpha}$ denotes the image of $\alpha$ via the natural $k$-homorphism $k\langle \,x_{11},x_{2t},\ldots, x_{nt}\,\rangle\to A_{nt}$. Since for any $i> j$, $x_{it}x_{jt}-r_t x_{jt}x_{it}$ is irreducible, $A_{nt}$ is a domain.   

\begin{lemma}\label{5.1} The following statements hold:
\begin{enumerate}	
\item $\GKdim_kA_{nt}=n$.
\item $A_{nt}$ is an Ore domain, the quotient division ring $D(A_{nt})$  of $A_{nt}$ is weakly locally finite, and $\GKdim_kD(A_{nt})=n$. 
\end{enumerate} 
\end{lemma}

\begin{proof} (1) Since $A_{nt}$ is finitely generated over $k$, by a remark in \cite[page 14]{zhang}, $\GKdim_kA_{nt}=\overline{\lim\limits_{r\to \infty}} {\log_r}{\dim_k}{V^r}$ for any subframe $V$ of $A_{nt}$. Let $$V=\langle \overline{1},\overline {x_{1t}}, \overline{x_{2t}},\ldots, \overline {x_{nt}}\rangle_k$$ be the vector subspace of $A_{nt}$ generated by $\{\,\overline 1,\overline {x_{1t}}, \overline {x_{2t}},\ldots, \overline {x_{nt}}\,\}$ over $k$. It is clear that $V$ is a subframe of $A_{nt}$. Now, we have $V^r=\langle\,\overline f\mid f\in B_r\,\rangle_k$,  where $B_r$ is the set of monomials $\overline{x_{j_1t}}. \overline{x_{j_2t}}\ldots \overline{x_{j_st}}$ of length $s\le r$ and $j_i\le j_{i+1}$ (notice that the monomial of length $0$ is $\overline{1}$). Observe that $B_r$ is independent over $k$ and the cardinality of $B_r$ is $$C^0_{n-1}+C^1_{n+1-1}+C^2_{n+2-1}+\ldots +C^r_{n+r-1}=C^{r+1}_{n+r}.$$ This  means that $\dim_k V^r=C^{r+1}_{n+r}$. Hence, $$\GKdim_kA_{nt}=\overline{\lim\limits_{r \to \infty}}{\log_r}{\dim_k}{V^r}=\overline{\lim\limits_{r \to \infty}}{\log_r}{C^{r+1}_{n+r}}=n.$$ 

(2)  Since $\GKdim_kA_{nt}=n$, the algebra $A_{nt}$ is an Ore domain by \cite[propositions 3.2 and 2.1]{zhang}. Denote by $D_{nt}=D(A_{nt})$ the quotient division ring of $A_{nt}$.  We will prove that  $D_{nt}$  is weakly locally finite. Indeed, put $k_t=k(\overline{x_{1t}}^{p^{2t}}, \overline{x_{2t}}^{p^{2t}},\ldots, \overline{x_{nt}}^{p^{2t}})$, the division subring of $D_{nt}$ generated by $\overline{x_{1t}}^{p^{2t}}, \overline{x_{2t}}^{p^{2t}},\ldots, \overline{x_{nt}}^{p^{2t}}$ over $k$, one has $x_{it}x_{jt}^{p^{2t}} = r_t^{p^{2t}}x_{jt}^{p^{2t}}x_{it} = x_{jt}^{p^{2t}}x_{it}$ for any $1\le j<i\le n$, which implies that $\overline{x_{it}}^{p^{2t}}$ belongs to the center $Z(D_{nt})$ of $D_{nt}$, and, consequently, $k_t$ is contained in $Z(D_{nt})$. Let $B_{nt}=k_t[\overline{x_{1t}},\overline{x_{2t}},\ldots, \overline{x_{nt}}]$ be the subring of $D_{nt}$ generated by $\overline{x_{11}},\overline{x_{2t}},\ldots, \overline{x_{nt}}$ over $k_t$. Then $A_{nt}\subseteq B_{nt}$. Observe that every element $f$ of $B_{nt}$ is of the form 
$$f = \sum\limits_{0\leq \mu_i\leq p^{2t}} a_{(\mu_1, \ldots, \mu_n)} \overline{x_{1t}}^{\mu_1} \ldots \overline{x_{nt}}^{\mu_n},$$
where $ a_{(\mu_1, \ldots, \mu_n)}\in k_t$. This implies that $B_{nt}$ is a finite-dimensional vector space over $k_t$. Thus, $B_{nt}$ can be considered as a subring of $M_m(k_t)$ with $m=\dim_{k_t}B_{nt}$. Since $M_m(k_t)$ is a PI-algebra, so are $B_{nt}$ and $A_{nt}$. Applying \cite[Theorem 5.6]{zhang}, one has $D_{nt}$ is locally PI and $\GKdim_kD_{nt}=n$. Now in view of Theorem~\ref{thm:2.2}, $D_{nt}$ is weakly locally finite.
\end{proof}

For a pair $(n, t)$ of positive integers, consider the $k$-homomorphism 
$$\phi_{nt}: k\langle x_{1(t-1)},x_{2(t-1)},\ldots,x_{n(t-1)}\rangle\to k\langle x_{1t},x_{2t},\ldots,x_{nt}\rangle,$$
defined by $\phi_{nt}(x_{i(t-1)})=x_{it}^p$ for any $1\le i\le n$. Then, $\phi_{nt}$ induces a $k$-homomorphism from $A_{n,t-1}$ to $A_{nt}$. In fact, we have the following lemma.  

\begin{lemma}\label{5.3} The $k$-homomorphism $\phi_{nt}$ induces the injective $k$-homomorphism 
$$\Phi_{nt}: A_{n,t-1}\to A_{nt},$$ with $\Phi_{nt} (\overline{x_{i(t-1)}})=\overline{x_{it}}^p$, for any $1\le i\le n$.
\end{lemma} 
\begin{proof} The most important thing in the proof of this lemma is to check that $\Phi$ is well-defined. To do this, we will show that for any $1\le j<i\le n$, the image $\phi(x_{i(t-1)}x_{j(t-1)}-r_{t-1}x_{j(t-1)}x_{i(t-1)})\in \langle\,x_{it}x_{jt}-r_tx_{jt}x_{it}\mid 1\le j<i\le n \,\rangle$. Indeed, for any $1\le j<i\le n$, we have 
$$\phi(x_{i(t-1)}x_{j(t-1)}-r_{t-1}x_{j(t-1)}x_{i(t-1)})=x^p_{it}x^p_{jt}-r_{t-1}x^p_{jt}x^p_{it}=x^p_{it}x^p_{jt}-r_t^{p^2}x^p_{jt}x^p_{it}$$
$$=x^{p-1}_{it}(x_{it}x_{jt}-r_tx_{jt}x_{it})x^{p-1}_{jt}\in \langle\,x_{it}x_{jt}-r_tx_{jt}x_{it}\mid 1\le j<i\le n \,\rangle.$$
\end{proof}

Now, for a given positive integer $n$, we are ready to give an example of a division ring with the Gelfand-Kirillov dimension $n$.

\begin{theorem}\label{thm:5.3}
Let $A_n=\bigcup\limits_{t \ge 1}{A_{nt}}$. Then, $A_n$ is an Ore domain. Moreover, if $D_n=D(A_n)$ is the quotient division ring of $A_n$, then we have 
\begin{enumerate}
\item The center $Z(D_n)$ of $D_n$ is $k$.
\item $D_n=\bigcup\limits_{t\ge 1}D_{nt}$ is weakly locally finite.
\item $\GKdim_kD_n=n$.
\item $D_n$ is not algebraic over $Z(D_n)$. In particular, $D_n$ is not locally finite.
\end{enumerate}
\end{theorem}
\begin{proof} The algebra $A_n$ is an Ore domain by \cite[Lemma 5.4 (2)]{zhang}. 
	
(1) Since $k \subseteq Z(D_n)$, we have to show  $Z(D_n)\subseteq k$. It suffices to show that none of indeterminates $x_{is}$ occurs in $f$ for any $f\in Z(D_n)$. 
Assume that this is false. Without loss of generality, we can suppose that $x_{1s}^m$ occurs in $f$, where $m$ is a non-zero integer with the smallest  absolute value. 
Since $f\in Z(D_n)\subseteq D=\bigcup\limits_{t\ge 1}D_{nt}$, there exists a positive integer $t_0$ such that $f\in D_{nt_0}$. Hence, $f\in D_{nt}$ for any $t\ge t_0$, so $f\in Z(D_{nt})$ for any $t\ge t_0$. Using arguments as in the proof of Lemma~\ref{5.1}, we conclude that  $k_t=k(\overline{x_{1t}}^{p^{2t}}, \overline{x_{2t}}^{p^{2t}},\ldots, \overline{x_{nt}}^{p^{2t}})$ is the center of $D_{nt}$. 
In view of  Lemma~\ref{5.3}, the element $\overline {x_{1s}^m}\in A_{ns}$ can be considered as the element $\overline{x_{1t}}^{mp^{t-s}}$ in $A_{nt}$ (via homomorphisms $\Phi_{nt}$) for any $t\ge s$. Since $f\in k_t$ for any $t\ge \max\{t_0,s\}$, all powers of $\overline{x_{1t}}$  divide $p^{2t}$ for any such a $t$. In particular, $mp^{t-s}$ divides $p^{2t}$ for any $t\ge \max\{t_0,s\}$, which is a contradiction.

(2)  We have $D_n=\bigcup\limits_{t\ge 1}D_{nt}$ by \cite[Lemma 5.4 (2)]{zhang}. Now for any finite subset $G$ of $D_n$, there exists $t_G$ such that $G\subseteq D_{nt_G}$. Thus, the division subring of $D_n$ generated by $G$ is contained in $D_{nt}$, hence it is centrally finite by Lemma~\ref{5.1}. Therefore, $D_n$ is weakly locally finite. 
	
(3) In view of (2) together with \cite[Theorem 5.6 and Lemma 5.4 (1)]{zhang}, it follows that $\GKdim_kD_n=n$.
	
(4) The conclusion is clear since $x_{11}$ is not algebraic over $Z(D_n)=k$.
\end{proof}

\subsection{Weakly locally finite division ring with the infinite Gelfand-Kirillov dimension}

For a pair $(n, t)$ of positive integers, consider the $k$-homomorphism 
$$\psi_{nt}: k\langle x_{1t},x_{2t},\ldots,x_{nt}\rangle\to k\langle x_{1t},x_{2t},\ldots,x_{n+1,t}\rangle,$$
defined by $\psi_{nt}(x_{nj})=x_{n+1, j}^p$ for any $1\le j\le t$. Then, $\psi_{nt}$ induces a $k$-homomorphism from $A_{nt}$ to $A_{n+1, t}$. In fact, we have the following lemma.  

\begin{lemma}\label{lem:6.1} The $k$-homomorphism $\psi_{nt}$ induces the injective $k$-homomorphism 
$$\Psi_{nt}: A_{nt}\to A_{n+1, t},$$ with $\Psi_{nt} (\overline {x_{nj}})=\overline {x_{n+1, j}}^p$, for any $1\le j\le t$.
\end{lemma} 
\begin{proof}
The proof of this lemma is similar to the proof of Lemma \ref{5.3}. 
\end{proof}

Put $A=\bigcup \limits_{n \ge 1} A_n$. Then, in view of \cite[Lemma 5.4]{zhang}, $A$ is an Ore domain, and $D=\bigcup\limits_{n \ge 1} D_n$ is the quotient division ring of $A$. By the same arguments as in the proof of Theorem \ref{thm:5.3}, we get the following result.

\begin{theorem}\label{thm:6.1} Let $A$ and $D$ be as above. Then, the following statements hold:
\begin{enumerate}
\item The center $Z(D)$ of $D$ is $k$.
\item $D$ is weakly locally finite.
\item $\GKdim_kD=\infty$.
\item $D$ is not algebraic over $Z(D)$. In particular, $D$ is not locally finite.
\end{enumerate}
\end{theorem}
\begin{proof} The proofs of (1), (2) and (4) are similar to that in the proof of Theorem ~\ref{thm:5.3}. So, it remains to prove (3). In fact, we have $\GKdim_kD\ge \GKdim_kD_n=n$ for any $n$ by Theorem \ref{thm:5.3}. Hence, $\GKdim_kD=\infty$.
\end{proof}

\section{Some facts related with the Kurosh Problem}

The Kurosh Problem for division rings we have mentioned in the Introduction can be formulated as the following.

\begin{problem}\label{problem:5.1}
Is it true that  every algebraic division ring is locally finite?
\end{problem}

The following theorem shows that the Kurosh problem is solved in the affirmative for the class of weakly locally finite division rings.

\begin{theorem}\label{thm:7.1} A division ring $D$ is locally finite if and only if $D$ is weakly locally finite and algebraic.
\end{theorem}

\begin{proof} If $D$ is locally finite, then clearly $D$ is both weakly locally finite and algebraic. Conversely, assume that $D$ is both weakly locally finite and algebraic.  Let $F=Z(D)$ and $S$ be a finite subset of $D$.
Since $D$ is weakly locally finite, the division subring $L$ of $D$ generated by $S$ is centrally finite.
Let $\mathcal{B}=\{x_1, x_2, \ldots, x_n\}$ be the basis of $L$ over its center  $Z(L)$.
For any $1\leq i, j\leq n$, write $x_{i}x_j= a_{ij1} x_1+a_{ij2}x_2+\ldots +a_{ijn}x_n,$ where $a_{ijk}\in Z(L)$.
Let $K$ be the division subring of $D$ generated by $F$ and all $a_{ijk}$. One has $K$ is a subfield of $D$.
By $D$ is algebraic over $F$ and set of all $a_{ijk}$ is finite, $K/F$ is a finite field extension.
	
Let $H=\{\,a_1 x_1+\ldots+ a_n x_n \mid a_i\in K\,\}$.
Then $H$ is a finite dimensional vector space over $K$, and it is clear that $H$ is a subring of $D$.
Now, for any $x\in H$, the set $\{1, x, x^2, \ldots, x^{n+1}\}$ is  linearly dependent over $K$, hence $\sum_{i=0}^n c_i x^i=0$ for some $c_i\in K$ not all zero. It follows that $x^{-1}\in H$, so $H$ is a division subring of $D$.
Moreover $\dim_F H<\infty$, since $\dim_K H<\infty$ and $K/F$ is a finite field extension.
	
It is easy to see that $H=F(S)$, and the proof is now complete.
\end{proof}

Note that the following  weaker version of the Kurosh Problem  is still open (see \cite[Problem 8]{zelmanov}).
\begin{problem}\label{problem:5.3}
Do there exist centrally infinite finitely generated division rings?
\end{problem}

If $D^*$ is finitely generated, then $D$ is finitely generated as a division ring. The converse may not be true. This fact leads us to consider the following problem which is also a weaker version of Problem \ref{problem:5.3}.

\begin{problem}\label{problem:5.4}
Do there exists a centrally infinite division ring $D$ whose multiplicative group  $D^*$ is finitely generated?  
\end{problem}

We devote the remaining part of the present section to the study of the matrix version of the last problem. More exactly, the following problem is under our consideration.

\begin{problem}\label{problem:5.5}
Do there exists a centrally infinite division ring $D$ such that the group $\mathrm{GL}_n(D), n\geq 1$ is finitely generated? 
\end{problem}

In the following, we shall identify $F^*$ with $F^*I:=\{\alpha I\vert~ \alpha\in F^*\}$, where $I$ denotes the identity matrix in $\mathrm{GL}_n(D)$.

In \cite{hdb}, we proved that if $D$ is a division ring of type $2$ and $D^*$ is finitely generated, then $D$ is a finite field. More generally, we showed that in a division ring of type $2$, there are no finitely generated non-central subgroups that contain the center $F^*$ (see \cite[Theorem 2.5]{hdb}). Recall that a division ring $D$ with center $F$ is of {\it type $2$} if for every two elements $x, y\in D$, the division subring $F(x, y)$ is a finite dimensional vector space over $F$. 

In \cite[Theorem 1]{mah1}, it was proved that if $D$ is centrally finite, then any finitely generated subnormal subgroup of $D^*$ is central. This result can be carried over for weakly locally finite division rings as the following.

\begin{theorem}\label{thm:9.1} Let $D$ be a weakly locally finite division ring. Then, every finitely generated subnormal subgroup of $D^*$ is central.
\end{theorem}
\begin{proof}  Since $N$ is finitely generated and $D$ is weakly locally finite, the division subring generated by $N$, namely $L$, is centrally finite.  By \cite[Theorem 1]{mah1}, $N\subseteq Z(L)$. Consequently, $N$ is abelian. Now, by \cite[14.4.4, p. 440]{scott}, $N\subseteq Z(D)$.
\end{proof}

The following theorem is a generalization of \cite[Theorem 5]{akb3}.
\begin{theorem}\label{thm:9.2}
Let $D$ be a weakly locally finite division ring with center $F$ and $N$ be an infinite  subnormal subgroup of $\mathrm{GL}_n(D), n\geq 2$. If $N$ is finitely generated, then  $N\subseteq F$.
\end{theorem}
\begin{proof}
Suppose that $N$ is non-central. Then, by  \cite[Theorem 11]{mah2}, $\mathrm{SL}_n(D)\subseteq N$. So, $N$ is normal in  $\mathrm{GL}_n(D)$. Suppose that $N$ is generated by matrices $A_1, A_2, \ldots, A_k$ in $\mathrm{GL}_n(D)$ and $T$ is the set of all coefficients of all $A_j$. 
Since $D$ is weakly locally finite, the division subring $L$ generated by $T$ is centrally finite.  It follows that $N$ is a normal finitely generated subgroup of $\mathrm{GL}_n(L)$. By \cite[Theorem 5]{akb3}, $N\subseteq Z(\mathrm{GL}_n(L))$.
In particular,  $N$ is abelian and consequently, $\mathrm{SL}_n(D)$ is abelian, a contradiction. 
\end{proof}
\begin{lemma}\label{lem:8.1} 
Let $D$ be a division ring with center $F$. If $N$ is a subnormal subgroup of $D^*$, then $Z(N)=N\cap F$.
\end{lemma}

\begin{proof} 
	If $N$ is contained in $F$, then there is nothing to prove. Thus, suppose that $N$ is non-central. By [9, 14.4.2, p. 439], $C_D(N)=F$. Hence $Z(N)\subseteq N\cap F$. Since the inclusion $N\cap F\subseteq Z(N)$ is obvious, $Z(N)= N\cap F$.
\end{proof}
\begin{lemma}\label{lem:8.2}
	If  $D$ is a weakly locally  finite division ring, then $Z(D')$ is a torsion group.
\end{lemma}
\begin{proof} By Lemma  \ref{lem:8.1}, $Z(D')=D'\cap F$. For any $x\in Z(D')$, there exists some positive integer $n$ and some $a_i, b_i\in D^*, 1\leq i\leq n$, such that 
	$$x=a_1b_1a_1^{-1}b_1^{-1}a_2b_2a_2^{-1}b_2^{-1}\ldots a_nb_na_n^{-1}b_n^{-1}.$$
	Set $S:=\{a_i,b_i\mid 1\leq i\leq n\}$.
	Since $D$ is weakly locally finite, the division subring $L$ of $D$ generated by $S$ is centrally finite.
	Put $n=[L:Z(L)].$ 
	Since $x\in F$, $x$ commutes with every element of $S$. Therefore,  $x$ commutes with every element of  $L$, and consequently,  $x\in Z(L)$. So,  
	$$x^n=N_{L/Z(L)}(x)=N_{L/Z(L)}(a_1b_1a_1^{-1}b_1^{-1}a_2b_2a_2^{-1}b_2^{-1}\ldots a_nb_na_n^{-1}b_n^{-1})=1.$$ 
	Thus, $x$ is torsion.
\end{proof}
\begin{lemma}\label{lem:9.1}
	Let $D$ be a division ring and $n\geq 1$. Then, $Z(\mathrm{SL}_n(D))$ is a torsion group if and only if $Z(D')$ is  a torsion group.
\end{lemma}
\begin{proof}
The case $n=1$ is clear. So, we can assume that $n\geq 2$. Denote by $F$ the center of $D$ and by $I_n$ the identity matrix of degree $n$. By \cite[\S21, Theorem 1, p.140]{dra}, 
$$Z(\mathrm{SL}_n(D))=\big\{ dI_n \vert d\in F^* \text{ and } d^n\in D'\big\}.$$ 
If $Z(\mathrm{SL}_n(D))$ is a torsion group, then, for any $d\in Z(D')=D'\cap F$, $dI_n\in Z(\mathrm{SL}_n(D))$. It follows that $d$ is torsion.
Conversely, if $Z(D')$ is a torsion group, then, for any $A\in Z(\mathrm{SL}_n(D))$, $A =dI_n$ for some $d\in F^*$ such that $d^n\in D'$. It follows that $d^n$ is torsion. Therefore, $A$ is torsion.
\end{proof}

\begin{theorem}\label{thm:9.3}
Let $D$ be a non-commutative algebraic, weakly locally finite division ring with center $F$ and  $N$ be a subgroup of $\mathrm{GL}_n(D)$ containing $F^*, n\geq 1$.  Then $N$ is not finitely generated.
\end{theorem}
\begin{proof} Recall  that if a division ring $D$ is weakly locally finite, then $Z(D')$ is a torsion group (see  Lemma \ref{lem:8.2}). Therefore, by Lemma \ref{lem:9.1}, $Z(\mathrm{SL}_n(D))$ is a torsion group.
	
Suppose that there is  a finitely generated subgroup $N$ of $\mathrm{GL}_n(D)$ containing $F^*$. Clearly $N/N'$ is a finitely generated abelian group, where $N'$ denotes the derived subgroup of $N$. Then, in virtue of \cite[5.5.8, p. 113]{scott}, $F^*N'/N'$ is a finitely generated abelian group.\\
	
\noindent{\em Case 1: $\mathrm{char}(D)=0$.}
	
Then, $F$ contains the field $\Q$ of rational numbers and it follows that $\Q^*I/(\Q^*I\cap N')\cong \Q^*N'/N'$. Since   $F^*N'/N'$ is finitely generated abelian subgroup, $\Q^*N'/N'$ is finitely generated too, and consequently   $\Q^*I/(\Q^*I\cap N')$ is finitely generated. Consider an arbitrary $A\in \Q^*I\cap N'$. Then $A\in F^*I\cap \mathrm{SL}_n(D)\subseteq Z(\mathrm{SL}_n(D))$. 
Therefore $A$ is torsion.
Since $A\in \Q^*I$, we have $A=dI$ for some $d\in \Q^*$. It follows that $d=\pm{1}$. Thus, $\Q^*I\cap N'$ is finite. Since $\Q^*I/(\Q^*I\cap N')$ is finitely generated, $\Q^*I$ is finitely generated. Therefore $\Q^*$ is finitely generated, that  is impossible.\\
	
\noindent{\em Case 2: $\mathrm{char}(D)=p > 0$.}
	
Denote by $\F_p$ the prime subfield of $F$, we shall prove that $F$ is algebraic over $\F_p$. In fact, suppose that $u\in F$ and $u$  is transcendental over $\F_p$. Put $K:=\F_p(u)$, then the group $K^*I/(K^*I\cap N')$ considered as a subgroup of $F^*N'/N'$ is finitely generated. Considering an arbitrary $A\in K^*I\cap N'$, we have $A=(f(u)/g(u))I$ for some $f(X), g(X)\in \F_p[X], ((f(X), g(X))=1$ and $g(u)\neq 0$. As mentioned above, we have $f(u)^s/g(u)^s=1$ for some positive integer $s$. Since $u$ is transcendental over $\F_p$,  $f(u)/g(u)\in \F_p$. Therefore, $K^*I\cap N'$ is finite and consequently, $K^*I$ is finitely generated. It follows that $K^*$ is finitely generated, hence $K$ is finite. Hence $F$ is algebraic over $\F_p$ and it follows that $D$ is algebraic over $\F_p$. Now, in virtue of Jacobson's Theorem  \cite[(13.11), p. 208]{lam},  $D$ is commutative, a contradiction. 
\end{proof}

\begin{corollary}\label{cor:9.1}
Let $D$ be an algebraic, weakly locally finite division ring. If the group $\mathrm{GL}_n(D), n\geq 1$, is finitely generated, then $D$ is commutative.
\end{corollary}

If $M$ is a maximal  finitely generated subgroup of $\mathrm{GL}_n(D)$, then $\mathrm{GL}_n(D)$ is finitely generated. So, the next result follows  immediately from  Corollary \ref{cor:9.1}.

\begin{corollary}\label{cor:9.2}
Let $D$ be an algebraic, weakly locally finite division ring. If the group $\mathrm{GL}_n(D), n\geq 1$,  has a maximal  finitely generated subgroup, then $D$ is commutative.
\end{corollary}

By the same way as in the proof of Theorem \ref{thm:9.3}, we obtain the following corollary.

\begin{corollary}\label{cor:9.3}
Let $D$ be a non-commutative  algebraic, weakly locally finite division ring with center $F$ and $S$ is a  subgroup of $\mathrm{GL}_n(D)$. If $N=F^*S$, then $N/N'$ is not finitely generated.
\end{corollary}

\begin{proof}
Suppose that $N/N'$ is finitely generated. Since $N'=S'$ and $ F^*I/(F^*I\cap S') \cong F^*S'/S'$, $F^*I/(F^*I\cap S')$ is a finitely generated abelian group.  Now, by the same arguments as in the proof of Theorem \ref{thm:9.3}, we conclude that $D$ is commutative.
\end{proof}

\begin{corollary}\label{cor:9.4}
Let $D$ be a non-commutative algebraic, weakly locally finite division ring. Then, $D^*$ is not finitely generated.
\end{corollary}

\begin{proof} Take $N=S=\mathrm{GL}_n(D)$ in Corollary \ref{cor:9.3} and have in mind that 
$$[\mathrm{GL}_n(D), \mathrm{GL}_n(D)]=\mathrm{SL}_n(D),$$ we see that $D^*\cong \mathrm{GL}_n(D)/\mathrm{SL}_n(D)$ is not finitely generated.
\end{proof}
  
\section{Herstein's conjecture for weakly locally finite division rings}

Let $K \varsubsetneq D$ be a pair of division rings. Recall that an element $x\in D$ is {\em radical} over $K$ if there exists some positive integer $n(x)$ depending on $x$ such that $x^{n(x)}\in K$. A subset $S$ of $D$ is {\em radical} over $K$ if every element from $S$ is radical over $K$. In 1978, I.N. Herstein  \cite[Conjecture 3]{her} conjectured that given a subnormal subgroup $N$ of $D^*$, if $N$ is radical over center $F$ of $D$, then $N$ is central, i. e. $N$ is contained in $F$. Herstein, himself in the cited above paper proved this fact for the special case, when $N$ is torsion group. However, the problem remains  still open in general. In \cite{hai-huynh}, it was proved that  this  conjecture is  true in the finite dimensional case. Here, we  shall prove that this conjecture  is also true for weakly locally finite division rings. 

In \cite[Theorem 1]{her}, Herstein proved that if in a division ring $D$ every multiplicative commutator $aba^{-1}b^{-1}$ is torsion, then $D$ is commutative. Further, with the assumption that $D$ is a finite dimensional vector space over its center $F$, he proved \cite[Theorem 2]{her} that, if every multiplicative commutator  in $D$ is radical over $F$, then $D$ is commutative. Now, using Lemma \ref{lem:8.2}, we can carry over  the last fact for weakly locally finite division rings.

\begin{theorem}\label{thm:8.1}
Let $D$ be a weakly locally  finite division ring with center $F$. If every multiplicative commutator in $D$ is radical over $F$, then  $D$ is commutative.
\end{theorem}
\begin{proof} For any $a, b\in D^*$, there exists a positive integer $n=n_{ab}$ depending on $a$ and $b$ such that $(aba^{-1}b^{-1})^n\in F.$ Hence, by  Lemma  \ref{lem:8.2}, it follows that $aba^{-1}b^{-1}$ is torsion. Now, by \cite[Theorem 1]{her}, $D$ is commutative.
\end{proof}

The following theorem gives the affirmative answer to Conjecture 3 in \cite{her} for  weakly locally finite division rings.

\begin{theorem}\label{thm:8.2}
Let $D$ be a weakly locally  finite division ring with center $F$ and  $N$ be a  subnormal subgroup of  $D^*$. If  $N$ is radical over $F$, then $N$ is central, i.e. $N$ is contained in $F$.
\end{theorem}
\begin{proof}
Consider the subgroup  $N'=[N,N]\subseteq D'$ and suppose that $x\in N'$. Since $N$ is radical over $F$, there exists some positive integer  $n$ such that $x^n\in F$. Hence $x^n \in F\cap D'=Z(D')$. By Lemma  \ref{lem:8.2},  $x^n$ is torsion, and consequently, $x$ is torsion too. Moreover, since  $N$ is subnormal in $D^*$, so is $N'$. Hence, by  \cite[Theorem 8]{her}, $N'\subseteq F$. Thus, $N$ is solvable, and by \cite[14.4.4, p. 440]{scott}, $N\subseteq F$.
\end{proof}

In Herstein's conjecture a subgroup $N$ is required to be radical over center $F$ of $D$. What happen if $N$ is required to be radical over some proper division subring of $D$ (which not necessarily coincides with $F$)? In the other words, the following question should be interesting: {\em ``Let $D$ be a division ring and $K$ be a proper division subring of $D$ and given a subnormal subgroup $N$ of $D^*$. If $N$ is radical over $K$, then is it contained in center $F$ of $D$?"} In the following we give the affirmative answer to this question for a weakly locally finite ring $D$ and a normal subgroup $N$. 

\begin{lemma}\label{lem:8.4} 
Let $D$ be a weakly locally  finite division ring with center $F$ and $N$ be a subnormal subgroup of $D^*$. If for every elements $x, y\in N$, there exists some positive integer $n_{xy}$ such that $x^{n_{xy}}y=yx^{n_{xy}}$, then $N\subseteq F$.
\end{lemma}
\begin{proof} 
Since $N$ is subnormal in $D^*$, there exists the following series of subgroups
$$N=N_1\vartriangleleft N_2\vartriangleleft\ldots\vartriangleleft N_r=D^*.$$
Suppose that $x, y\in N$. Let $K$ be the division subring of $D$ generated by $x$ and $y$. 
Then, $K$ is centrally finite.
By putting $M_i=K\cap N_i, \forall i\in\{1, \ldots, r\}$ we obtain the following series of subgroups
$$M_1\vartriangleleft M_2\vartriangleleft\ldots\vartriangleleft M_r=K^*.$$
For any $a\in M_1\leq N_1=N$, suppose that $n_{ax}$ and $n_{ay}$ are positive integers such that
$a^{n_{ax}}x=xa^{n_{ax}}$ and $a^{n_{ay}}y=ya^{n_{ay}}.$
Then, for $n:=n_{ax}n_{ay}$ we have 
$a^n=$ $(a^{n_{ax}})^{n_{ay}}$ $=(xa^{n_{ax}}x^{-1})^{n_{ay}}$ $=xa^{n_{ax}n_{ay}}x^{-1}$ $=xa^nx^{-1},$
and
$a^n$ $=(a^{n_{ay}})^{n_{ax}}$ $=(ya^{n_{ay}}y^{-1})^{n_{ax}}$ $=ya^{n_{ay}n_{ay}}y^{-1}$ $=ya^ny^{-1}.$
Therefore $a^n\in Z(K)$. Hence $M_1$ is radical over $Z(K)$. By  Theorem \ref{thm:8.2}, $M_1\subseteq Z(K)$. In particular, $x$ and $y$  commute with each other. Consequently, $N$ is abelian group. By \cite[14.4.4, p. 440]{scott},  $N\subseteq F$.
\end{proof}
\begin{theorem}\label{thm:8.3} 
Let $D$ be a weakly locally  finite division ring with center $F$ and $K$ be a proper division subring of $D$. Then, every normal subgroup of $D^*$ which is radical over $K$ is contained in $F$.
\end{theorem}

\begin{proof} 
Assume that $N$ is a normal subgroup of $D^*$ which is radical over $K$, and $N$ is not contained in the center $F$. If $N\setminus K=\emptyset$, then $N\subseteq K$. By \cite[p. 433]{scott}, either $K\subseteq F$ or $K=D$. Since $K\neq D$ by the assertion, it follows that $K\subseteq F$. Hence $N\subseteq F$, that contradicts to the assertion. Thus, we have $N\setminus K\neq\emptyset$.

Now, to complete the proof of our theorem we shall show that the elements of $N$ satisfy the requirements of Lemma \ref{lem:8.4}. Thus, suppose that  $a, b\in N$. We examine the following cases:\\

\noindent{\em Case 1:}  $a\in K$.

\noindent{\em  Subcase 1.1:} $b\not\in K$. 

We shall prove that there exists some positive integer $n$ such that $a^nb=ba^n$. Thus, suppose that $a^nb\neq ba^n$ for any positive integer $n$. Then, $a+b\neq 0, a\neq \pm{1}$ and $b\neq \pm{1}$. So we have
$$x=(a+b)a(a+b)^{-1}, y=(b+1)a(b+1)^{-1}\in N.$$
Since $N$ is radical over $K$, we can find some positive integers $m_x$ and $m_y$ such that
$$x^{m_x}=(a+b)a^{m_x}(a+b)^{-1}, y^{m_y}=(b+1)a^{m_y}(b+1)^{-1}\in K.$$
Putting $m=m_xm_y$, we have
$$x^m=(a+b)a^m(a+b)^{-1}, y^m=(b+1)a^m(b+1)^{-1}\in K.$$
Direct calculations give the equalities
$$x^mb-y^mb+x^ma-y^m=x^m(a+b)-y^m(b+1)=(a+b)a^m-(b+1)a^m=a^m(a-1),$$
from that we get the following equality
$$(x^m-y^m)b=a^m(a-1)+y^m-x^ma.$$
If $(x^m-y^m)\neq 0$, then $b=(x^m-y^m)^{-1}[a(a^m-1)+y^m-x^ma]\in K$, that is a contradiction to the choice of $b$. Therefore $(x^m-y^m)= 0$ and consequently, $a^m(a-1)=y^m(a-1)$. Since $a\neq 1,a^m=y^m=(b+1)a^m(b+1)^{-1}$ and it follows that $a^mb=ba^m$, a contradiction.

\noindent{\em  Subcase 1.2:} $b\in K$. 

Consider an element $x\in N\setminus K$. Since $xb\not\in K$, by Subcase 1.1, there exist some positive integers $r, s$ such that 
$a^rxb=xba^r$ and $a^sx=xa^s.$
From these equalities it follows that
$a^{rs}=(xb)^{-1}a^{rs}(xb)=b^{-1}(x^{-1}a^{rs}x)b=b^{-1}a^{rs}b,$
and consequently, $a^{rs}b=ba^{rs}.$\\

\noindent {\em Case 2:}  $a\not\in K$.

Since $N$ is radical over $K$, there exists some positive integer $m$ such that $a^m\in K$.  By Case 1, there exists some positive integer $n$ such that $a^{mn}b=ba^{mn}$.
\end{proof}

\begin{acknowledgements}
The authors would like to thank the referee for his/her comments and suggestions. Also, they sincerely thank Adrian Wadsworth and John McConnell for the conversations by e-mail about McConnell's construction of examples of locally PI, but not PI division rings of arbitrary GK-dimension.
\end{acknowledgements}

\end{document}